\documentclass[11pt]{amsart}
\usepackage[utf8]{inputenc}
\usepackage[margin=1.00in]{geometry}
\usepackage[english]{babel}
\usepackage{subcaption}
\usepackage{amsmath}
\usepackage{amssymb}
\usepackage{amsfonts}
\usepackage{amsthm}
\usepackage{mathtools}
\usepackage{soul}
\usepackage{mathdots}
\usepackage{mathrsfs}
\usepackage[backend=biber]{biblatex}

\AtEveryBibitem{
\ifentrytype{misc}{}{\clearfield{eprint}}
}
\usepackage{csquotes}
\usepackage[all]{xy}
\usepackage[pdftex]{graphicx}
\usepackage{color}
\usepackage{url}
\usepackage{enumitem}
\usepackage[labelfont=bf,labelsep=period,justification=raggedright]{caption}
\usepackage[english]{babel}
\usepackage[utf8]{inputenc}
\usepackage[colorlinks=true, citecolor=blue, linkcolor = blue]{hyperref}
\usepackage[colorinlistoftodos]{todonotes}
\usepackage{tkz-fct}
\usepackage{tikz}
\usepackage{tikz-3dplot}
\usetikzlibrary{calc}
\usetikzlibrary{knots}
\usetikzlibrary{spath3, intersections, hobby}
\usetikzlibrary{arrows, automata}
\usetikzlibrary{arrows.meta}
\usetikzlibrary{shapes.multipart}
\usepackage{pgfplots}
\usepackage{multicol}
\PassOptionsToPackage{dvipsnames,svgnames}{xcolor}
\usepackage{textcomp}
\usepackage{bbm}
\usepackage{array}
\newcolumntype{C}{>{$}c<{$}}
\usepackage{forest}
\usepackage{tikzsymbols}
\usetikzlibrary{3d}
\tdplotsetmaincoords{80}{110}
\usepackage{subfiles}
\usepackage{stmaryrd}
\makeatletter
\DeclareFontFamily{OMX}{MnSymbolE}{}
\DeclareSymbolFont{MnLargeSymbols}{OMX}{MnSymbolE}{m}{n}
\SetSymbolFont{MnLargeSymbols}{bold}{OMX}{MnSymbolE}{b}{n}
\DeclareFontShape{OMX}{MnSymbolE}{m}{n}{
    <-6>  MnSymbolE5
   <6-7>  MnSymbolE6
   <7-8>  MnSymbolE7
   <8-9>  MnSymbolE8
   <9-10> MnSymbolE9
  <10-12> MnSymbolE10
  <12->   MnSymbolE12
}{}
\DeclareFontShape{OMX}{MnSymbolE}{b}{n}{
    <-6>  MnSymbolE-Bold5
   <6-7>  MnSymbolE-Bold6
   <7-8>  MnSymbolE-Bold7
   <8-9>  MnSymbolE-Bold8
   <9-10> MnSymbolE-Bold9
  <10-12> MnSymbolE-Bold10
  <12->   MnSymbolE-Bold12
}{}

\let\llangle\@undefined
\let\rrangle\@undefined
\DeclareMathDelimiter{\llangle}{\mathopen}%
                     {MnLargeSymbols}{'164}{MnLargeSymbols}{'164}
\DeclareMathDelimiter{\rrangle}{\mathclose}%
                     {MnLargeSymbols}{'171}{MnLargeSymbols}{'171}
\makeatother

\makeatletter
\tikzoption{canvas is xy plane at z}[]{%
  \def\tikz@plane@origin{\pgfpointxyz{0}{0}{#1}}%
  \def\tikz@plane@x{\pgfpointxyz{1}{0}{#1}}%
  \def\tikz@plane@y{\pgfpointxyz{0}{1}{#1}}%
  \tikz@canvas@is@plane
}
\makeatother

\tikzset{surface/.style={draw=black, left color=orange,right color=orange,middle
color=orange!60!#1, fill opacity=1},surface/.default=white}

\tikzset{>=latex}

\pgfplotsset{compat=1.17}

\def\multiset#1#2{\ensuremath{\left(\kern-.3em\left(\genfrac{}{}{0pt}{}{#1}{#2}\right)\kern-.3em\right)}}

\providecommand{\Red}{\operatorname{Red}}

\makeatletter

\newif\ifpgfcirclecrosssplitcustomfill
\tikzset{%
  circle cross split part fill/.code=\def\pgf@lib@sh@ccs@list@fill{#1}\pgfcirclecrosssplitcustomfilltrue,%
  circle cross split uses custom fill/.is if=pgfcirclecrosssplitcustomfill}
\pgfdeclareshape{circle cross split}{%
  \nodeparts{text,two,three,four}%
  \savedanchor\centerpoint{%
    \pgfmathsetlength\pgf@xa{\pgfkeysvalueof{/pgf/inner xsep}}%
    \pgfmathsetlength\pgf@ya{\pgfkeysvalueof{/pgf/inner ysep}}%
    \pgf@x\wd\pgfnodeparttextbox
    \pgf@yb\dp\pgfnodeparttextbox
    \pgf@yc\dp\pgfnodeparttwobox
    \ifdim\pgf@yb>\pgf@yc
      \pgf@yc\pgf@yb
    \fi
    \advance\pgf@y-\pgf@yc
    \advance\pgf@x\pgf@xa
    \advance\pgf@y-\pgf@ya
    \advance\pgf@x.5\pgflinewidth
    \advance\pgf@y-.5\pgflinewidth
  }%
  \savedanchor\twoanchor{%
    \pgfmathsetlength\pgf@xa{\pgfkeysvalueof{/pgf/inner xsep}}%
    \pgfmathsetlength\pgf@ya{\pgfkeysvalueof{/pgf/inner ysep}}%
    \advance\pgf@x.5\pgflinewidth
    \advance\pgf@x\pgf@xa
    \advance\pgf@y.5\pgflinewidth
    \advance\pgf@y\pgf@ya
    \pgf@yb\dp\pgfnodeparttextbox
    \pgf@yc\dp\pgfnodeparttwobox
    \ifdim\pgf@yb>\pgf@yc
      \pgf@yc\pgf@yb
    \fi
    \advance\pgf@y\pgf@yc
  }%
  \savedanchor\threeanchor{%
    \pgfmathsetlength\pgf@ya{\pgfkeysvalueof{/pgf/inner ysep}}%
    \pgf@x\wd\pgfnodeparttextbox
    \pgf@yb\dp\pgfnodeparttextbox
    \pgf@yc\dp\pgfnodeparttwobox
    \ifdim\pgf@yb>\pgf@yc
      \pgf@yc\pgf@yb
    \fi
    \advance\pgf@y-\pgf@yc
    \advance\pgf@y-2\pgf@ya
    \advance\pgf@y-\pgflinewidth
    \pgf@yb\ht\pgfnodepartthreebox
    \pgf@yc\ht\pgfnodepartfourbox
    \ifdim\pgf@yb>\pgf@yc
      \pgf@yc\pgf@yb
    \fi
    \advance\pgf@y-\pgf@yc
    \advance\pgf@x-\wd\pgfnodepartthreebox
  }%
  \savedanchor\fouranchor{%
    \pgfmathsetlength\pgf@xa{\pgfkeysvalueof{/pgf/inner xsep}}%
    \advance\pgf@x\wd\pgfnodepartthreebox
    \advance\pgf@x2\pgf@xa
    \advance\pgf@x\pgflinewidth
  }%
  \saveddimen\radius{%
    \pgf@y\ht\pgfnodeparttextbox
    \pgf@yb\ht\pgfnodeparttwobox
    \ifdim\pgf@yb>\pgf@y
      \pgf@y\pgf@yb
    \fi
    \pgf@yc\dp\pgfnodeparttextbox
    \pgf@yb\dp\pgfnodeparttwobox
    \ifdim\pgf@yc>\pgf@yb
      \advance\pgf@y\pgf@yc
    \else
      \advance\pgf@y\pgf@yb
    \fi
    \pgf@yb\ht\pgfnodepartthreebox
    \ifdim\pgf@yb<\ht\pgfnodepartfourbox
      \pgf@yb\ht\pgfnodepartfourbox
    \fi
    \pgf@yc\dp\pgfnodepartthreebox
    \ifdim\pgf@yc<\dp\pgfnodepartfourbox
      \advance\pgf@yb\dp\pgfnodepartfourbox
    \else
      \advance\pgf@yb\pgf@yc
    \fi
    \ifdim\pgf@yc>\pgf@y
      \pgf@y\pgf@yc
    \fi
    \pgfmathsetlength\pgf@ya{\pgfkeysvalueof{/pgf/inner ysep}}%
    \advance\pgf@y2\pgf@ya
    \pgf@x\wd\pgfnodeparttextbox
    \pgf@xa\wd\pgfnodepartthreebox
    \pgf@xb\wd\pgfnodeparttwobox
    \pgf@xc\wd\pgfnodepartfourbox
    \ifdim\pgf@xa>\pgf@x
      \pgf@x\pgf@xa
    \fi
    \ifdim\pgf@xb>\pgf@x
      \pgf@x\pgf@xb
    \fi
    \ifdim\pgf@xc>\pgf@x
      \pgf@x\pgf@xc
    \fi
    \pgfmathsetlength\pgf@xa{\pgfkeysvalueof{/pgf/inner xsep}}%
    \advance\pgf@x2\pgf@xa
    \ifdim\pgf@y>\pgf@x
      \pgf@x\pgf@y
    \fi
    \advance\pgf@x.5\pgflinewidth
    \pgfmathsetlength{\pgf@xb}{\pgfkeysvalueof{/pgf/minimum width}}%
    \pgfmathsetlength{\pgf@yb}{\pgfkeysvalueof{/pgf/minimum height}}%
    \ifdim\pgf@x<.5\pgf@xb
        \pgf@x=.5\pgf@xb
    \fi
    \ifdim\pgf@x<.5\pgf@yb
        \pgf@x=.5\pgf@yb
    \fi
    \pgfmathsetlength{\pgf@xb}{\pgfkeysvalueof{/pgf/outer xsep}}%
    \pgfmathsetlength{\pgf@yb}{\pgfkeysvalueof{/pgf/outer ysep}}%
    \ifdim\pgf@xb<\pgf@yb
      \advance\pgf@x\pgf@yb
    \else
      \advance\pgf@x\pgf@xb
    \fi
  }%
  \inheritanchorborder[from=circle]%
  \inheritanchor[from=circle]{north}%
  \inheritanchor[from=circle]{north west}%
  \inheritanchor[from=circle]{north east}%
  \inheritanchor[from=circle]{center}%
  \inheritanchor[from=circle]{west}%
  \inheritanchor[from=circle]{east}%
  \inheritanchor[from=circle]{mid}%
  \inheritanchor[from=circle]{mid west}%
  \inheritanchor[from=circle]{mid east}%
  \inheritanchor[from=circle]{base}%
  \inheritanchor[from=circle]{base west}%
  \inheritanchor[from=circle]{base east}%
  \inheritanchor[from=circle]{south}%
  \inheritanchor[from=circle]{south west}%
  \inheritanchor[from=circle]{south east}%
  \anchor{two}{\twoanchor}%
  \anchor{three}{\threeanchor}%
  \anchor{four}{\fouranchor}%
  \inheritbackgroundpath[from=circle]
  \beforebackgroundpath{%
    \pgfutil@tempdima=\radius
    \pgfmathsetlength{\pgf@xb}{\pgfkeysvalueof{/pgf/outer xsep}}%
    \pgfmathsetlength{\pgf@yb}{\pgfkeysvalueof{/pgf/outer ysep}}%
    \ifdim\pgf@xb<\pgf@yb
      \advance\pgfutil@tempdima by-\pgf@yb
    \else
      \advance\pgfutil@tempdima by-\pgf@xb
    \fi
    \advance\pgfutil@tempdima by-.5\pgflinewidth%
    \pgfsetshortenstart{0pt}%
    \pgfsetshortenend{0pt}%
    \pgfsetarrows{-}%
    \pgfpathmoveto{\pgfpointadd{\centerpoint}{\pgfqpoint{-\pgfutil@tempdima}{0pt}}}%
    \pgfpathlineto{\pgfpointadd{\centerpoint}{\pgfqpoint{\pgfutil@tempdima}{0pt}}}%
    \pgfpathmoveto{\pgfpointadd{\centerpoint}{\pgfqpoint{0pt}{-\pgfutil@tempdima}}}%
    \pgfpathlineto{\pgfpointadd{\centerpoint}{\pgfqpoint{0pt}{\pgfutil@tempdima}}}%
    \pgfusepathqstroke
  }%
  \behindbackgroundpath{%
    \pgfutil@tempdima=\radius
    \pgfmathsetlength{\pgf@xb}{\pgfkeysvalueof{/pgf/outer xsep}}%
    \pgfmathsetlength{\pgf@yb}{\pgfkeysvalueof{/pgf/outer ysep}}%
    \ifdim\pgf@xb<\pgf@yb
      \advance\pgfutil@tempdima by-\pgf@yb
    \else
      \advance\pgfutil@tempdima by-\pgf@xb
    \fi
    \advance\pgfutil@tempdima by-.5\pgflinewidth%
    \ifpgfcirclecrosssplitcustomfill%
      \pgf@lib@sh@rs@process@list{\pgf@lib@sh@ccs@list@fill}{4}%
      {%
        \pgfmathloop
           \ifnum\pgfmathcounter>4%
           \else%
             \pgf@lib@sh@getalpha\pgf@lib@sh@rs@number{\pgfmathcounter}%
              \edef\pgf@tempa{\csname pgf@lib@sh@rs@\pgf@lib@sh@rs@number @item\endcsname}%
              \ifx\pgf@tempa\pgf@lib@sh@rs@nonetext\else
                \pgfsetfillcolor{\pgf@tempa}%
                \pgf@lib@sh@ccs@angles{\pgfmathcounter}%
                \pgfpathmoveto{\centerpoint}%
                \pgfpathlineto{\pgfpointadd{\centerpoint}{\pgfqpointpolar{\pgf@lib@sh@ccs@angle}{\pgfutil@tempdima}}}%
                \pgfpatharc{\pgf@lib@sh@ccs@angle}{\pgf@lib@sh@ccs@angle@}{\pgfutil@tempdima}%
                \pgfpathclose
                \pgfusepathqfill
              \fi
        \repeatpgfmathloop
      }%
    \fi%
  }%
}
\def\pgf@lib@sh@ccs@angles#1{%
  \ifcase#1\or\def\pgf@lib@sh@ccs@angle{90}%
           \or\def\pgf@lib@sh@ccs@angle{0}%
           \or\def\pgf@lib@sh@ccs@angle{180}%
           \else\def\pgf@lib@sh@ccs@angle{270}%
  \fi
  \edef\pgf@lib@sh@ccs@angle@{\number\numexpr\pgf@lib@sh@ccs@angle+90\relax}%
}
\makeatother

\theoremstyle{plain}
\newtheorem{thm}{Theorem}
\newtheorem{lemma}[thm]{Lemma}
\newtheorem{cor}[thm]{Corollary}

\newtheorem{prop}[thm]{Proposition}

\newtheorem{qn}[thm]{Question}
\newtheorem{claim}[thm]{Claim}

\theoremstyle{definition}
\newtheorem{defn}[thm]{Definition}

\theoremstyle{remark}
\newtheorem*{rem*}{Remark}

\newtheorem*{prop*}{Proposition}
\newtheorem*{claim*}{Claim}

\numberwithin{equation}{section}
\numberwithin{thm}{section}
\makeatletter
\newcommand{\tpmod}[1]{{\@displayfalse\pmod{#1}}}
\makeatother

\title{On the distance distribution of Reed-Muller codes}
\author{Neil Kolekar}
\email{nkolekar@gmail.com}
\date{\today}

\begin{document}
\setlength{\parindent}{15pt}

\begin{abstract}
     In this paper, we give error bounds for the distance distribution of Reed-Muller codes, extending prior work on the distance distribution of Reed-Solomon codes. This is equivalent to the problem of counting multivariate polynomials over a finite field with prescribed degree, coefficients, and number of zeroes. We provide a solution to this problem using the character sum method, which offers a new unified framework applicable to a broad class of polynomial enumeration problems over finite fields that involve prescribed evaluation vectors. 
     
     This work effectively makes the first systematic attempt to study the coset weight distribution problem for Reed-Muller codes of fixed degree over large finite fields, which was proposed in MacWilliams and Sloane's 1977 textbook \emph{The Theory of Error Correcting Codes}. 
\end{abstract}

\maketitle

\section{Introduction}
\subsection{Motivating problem}
Let $\mathbb{F}_q$ be the finite field of order $q$ and characteristic $p$. Fix a subset $D = \{\alpha_1, \dots, \alpha_n\}\subseteq \mathbb{F}_q^v$. The \emph{Reed-Muller code} $\textrm{RM}_q(d, v)$ is the set of all codewords of the form \[(f(\alpha_1), \dots, f(\alpha_n)),\]where $f$ ranges over polynomials in $\mathbb{F}_q[x_1, \dots, x_v]$ of degree at most $d-1$. 

For two vectors $\mathbf{u} = (u_1, \dots, u_k)$ and $\mathbf{w} = (w_1, \dots, w_k)$ in $\mathbb{F}_q^k$, the \emph{Hamming distance} $\operatorname{dist}(\mathbf{u}, \mathbf{w})$ is the number of indices $1 \leq i \leq k$ such that $u_i \neq w_i$. Accordingly, if $f, g \in \mathbb{F}_q[x_1, \dots, x_v]$ are polynomials of degree at most $d-1$, then the Hamming distance between their associated codewords equals $n-r$, where $r$ is the number of common evaluation points $\alpha_i \in D$ at which $f(\alpha_i) = g(\alpha_i)$, that is, the number of zeroes of $f-g$ on $D$. The \emph{standard Reed-Muller codes} are those with $D=\mathbb{F}_q^v \setminus \{\mathbf{0}\}$. The \emph{distance distribution problem} for Reed-Muller codes asks the following. 
\begin{qn}\label{qn: distance distribution problem}
    Given a received word $\mathbf{u}$ and positive integer $k$, compute the number of codewords $\mathbf{w} \in \mathrm{RM}_q(d, v)$ for which $\operatorname{dist}(\mathbf{u}, \mathbf{w}) = k$.
\end{qn}
Question \ref{qn: distance distribution problem}, in the context of general linear codes, is an important problem in coding theory which is widely open for most codes. The distance distribution problem essentially consists of two subproblems. In the case where $\mathbf{u} \in \mathrm{RM}_q(d, v)$, the problem becomes the \emph{weight distribution problem}; when $\mathbf{u} \not\in \mathrm{RM}_q(d, v)$, the problem becomes the \emph{coset weight distribution problem}. Our interest mainly lies in the latter, which has seen little progress for Reed-Muller codes since it was introduced in the 1970s. The coset weight distribution problem dates back to 1977 when it was proposed in MacWilliams and Sloane's classic textbook \emph{The Theory of Error-Correcting Codes} \cite[, Research Problem 1.1]{MacWilliams1977}; the weight distribution problem was proposed in the same text \cite[, Research Problem 15.1]{MacWilliams1977}. Even in the specialization of Reed-Solomon codes, the coset weight distribution problem has only recently been studied \cite{LiWan2020}. In this paper, we study the distance distribution problem for Reed-Muller codes, taking forward the character sum approach of Li and Wan \cite{LiWan2020} to the multivariate setting. \\

Naturally, one would consider the polynomial interpretation of this problem. The equivalent translation is presented below, upon taking the polynomial reversal of the polynomials associated with the (code)words involved.

\begin{qn} 
Fix a positive integer $\ell \le v(q-1)$. Given a fixed polynomial $g \in \mathbb{F}_q[x_1, \dots, x_v]$ of degree $\ell-1$, compute the number of polynomials $f \in \mathbb{F}_q[x_1, \dots, x_v]$ such that the maximum degree of $x_i$ in $f$ is at most $q-1$ for all $1 \le i \le v$, $f-g$ has exactly $s$ zeroes in $\mathbb{F}_q^v$, and \[ \overline{f-g} \in \langle x_1^{i_1}\dots x_v^{i_v} \mid i_1+\dots+i_v = \ell\rangle, \] where $\overline{f-g}$ is defined as the polynomial $x_1^{q-1}\cdots x_v^{q-1}(f-g)(x_1^{-1}, \dots, x_v^{-1})$. 
\end{qn}

The primary goal of this paper is to give error bounds on the count above, which produces error bounds for \emph{standard} Reed-Muller codes. Our main result is the following.

\begin{thm}\label{main result}
    Let $N_k(d, g)$ denote the number of polynomials $f(x_1, \dots, x_v) \in \mathbb{F}_q[x_1, \dots, x_v]$ such that the degree of every variable $x_i$ in $f$ is at most $q-1$, $f-g$ has no nonzero terms of degree $\ell-1$ or lower, and $\overline{f-g}$ has exactly $q^v-k-1$ zeroes. Furthermore, let \[ E=\#\{(i_1, \dots, i_v) \in (\mathbb{Z} \cap [0, q-1])^v \mid i_1+\dots i_v \le \ell-1\}. \] Then, $N_k(d, g)$ satisfies the error bound \[ \left|N_k(d,g)-\binom{(q-1)q^{v-1}}{k}\, q^{k-E-q^{v-1}+1}\right| \le \binom{v q_2 q^{\binom{\ell+v-3}{v-1}}+k-1}{k-1} q^{k\left(\binom{\lceil \ell/2 \rceil + v}{v}+1\right)-E},
\] where $q_1 = \max(\ell-1, (q-1)/\sqrt{q(p-1)})$ and \[ q_2 = q_1 + \left(\frac{q-1}{\sqrt{p-1}}-q_1\right) \cdot \left(\sum_{k=1}^{\lfloor p/2 \rfloor} \binom{q}{k} \cdot \left(\frac{\ell}{p}\right)^k\right)^{-1}. \]
\end{thm}

 \subsection{Past work}
     In coding theory literature, the distance distribution problem for linear codes dates back to 1977 \cite{MacWilliams1977}. There is, however, a related problem called the \emph{list-decoding problem} which was introduced in 1957 by Elias \cite{Elias1957}. Instead of asking for the number of codewords which are precisely some distance away from a particular code, the list-decoding problem asks for an algorithm to determine the list of codewords \emph{within} that distance. Beginning from the 1960s, there has been much progress made on the algorithmic aspect of the list-decoding problem. On the other hand, between 2007 and 2015 there has been much work on the combinatorial aspect of this problem \cite{Bhowmick2014, Gopalan2010, Kaufman2008}, which eventually resolved the weight distribution problem for Reed-Muller codes \cite{Bhowmick2022}. The techniques used for the work in this time period were largely combinatorial, in that they were based on the higher-order Fourier analysis of multivariate polynomials. The natural next step in this line of research is to consider the distance distribution problem. Currently, the progress made towards the goal of resolving this problem for the generality of Reed-Muller codes only applies to the specialization of Reed-Solomon codes. 
 
  In 2020, Li and Wan \cite{LiWan2020} gave error bounds on the number of Reed-Solomon codewords that are a fixed distance away from a given codeword, using the character sum method, improving error bounds of Zhu and Wan \cite{Zhu2012}. These error bounds were later improved upon by Gao and Li \cite{Gao2022} in 2022, who strengthened a weak bound in \cite{LiWan2020} using the saddle point method. The improved error bounds were used towards constructing new classes of deep holes of Reed-Solomon codes. Before the character sum method of Zhu and Wan for estimating the distance distribution was introduced, the techniques that were employed for Reed-Solomon codes were highly algebro-geometric. These techniques led to the results in \cite{Cafure2011}, \cite{Cheng2005}, and \cite{Liao2010}.

  Like past work on the distance distribution of Reed-Solomon codes, this paper too utilizes character sums. Specifically, we understand the vanishing sets of polynomials using Lagrange interpolation, and analyze character sums over truncated multivariate polynomial quotient rings involving both additive and multiplicative characters. The polynomials associated with Lagrange interpolation are multiplicatively well-structured, but are not well-structured in the additive sense. This motivates combining additive characters and multiplicative characters to estimate the desired number of codes. In particular, we evaluate Gauss sums over truncated polynomial quotient rings, which have deep connections to the algebraic structure of these rings. Part of the technical work we perform pertaining to algebraic properties of these quotient rings has not been studied prior to this paper, and is thus potentially useful beyond the scope of the distance distribution problem. 
 
 \subsection*{Outline}
 The remainder of this paper is structured as follows. In Section \ref{section: preliminaries}, we set up the notation and terminology used in this paper. In Section \ref{section: char sum setup}, we use Lagrange interpolation to establish a character sum formula for the count that we aim to give error bounds on. Section \ref{section: center of approximation} provides an exact formula for the center of approximation we use. In Section \ref{section: gauss sums}, we find an explicit formula for the sizes of the Gauss sums over a particular polynomial quotient ring. In Section \ref{section: key enumerations}, we perform two key enumerations required for our bounds: the first is of ideals of the polynomial quotient ring with certain properties, and the second is of multiplicative characters with prescribed values. To that end, {our main result} is proved in Section \ref{section: li-wan sieve}, using the Li-Wan coordinate-sieve formula aided by our results leading up to that point.  Finally, we close the paper with Section \ref{section: future directions} in which we discuss potential future research directions.

\section{Preliminaries}\label{section: preliminaries}
In this section, we establish the notation and terminology used throughout the paper. \\

We begin with notation for polynomials over finite fields. Let $\mathbb{F}_q$ be the finite field of order $q$ and characteristic $p$, and let $\mathbb{F}_q^{*} = \mathbb{F}_q \setminus \{0\}$. Write $\mathbb{F}_q[x_1, \dots, x_v]$ for the polynomial ring in variables $x_1, \dots, x_v$ over $\mathbb{F}_q$. For $f \in \mathbb{F}_q[x_1, \dots, x_v]$, we use  
\[
[x_1^{a_1}\cdots x_v^{a_v}]f
\]
to denote the coefficient of $x_1^{a_1}\cdots x_v^{a_v}$ in $f$. We write $\mathbf{x} = (x_1, \dots, x_v)$, so that $f(\mathbf{x})$ abbreviates $f(x_1, \dots, x_v)$. Given a subset $T \subseteq [v]$, denote $x_T$ by the monomial $\prod_{i \in T} x_i$. 

The degree of a polynomial $f$ is denoted by $\deg(f)$, with the convention $\deg(0) = -\infty$. Let $\mathcal{M}$ denote the set of polynomials $f \in \mathbb{F}_q[x_1, \dots, x_v]$ in which the degree of each variable $x_i$ is at most $q-1$. Finally, for $f \in \mathcal{M}$, define $\overline{f} = x_1^{q-1}\cdots x_v^{q-1}f(x_1^{-1}, \dots, x_v^{-1})$.

Set 
\[\mathcal{E} := \frac{\mathbb{F}_q[x_1, \dots, x_v]}{\langle x_1^{i_1} \dots x_v^{i_v} \mid i_1+\dots+i_v = \ell \rangle}.\] Note that the equivalence classes defined by the quotient ring $\mathcal{E}$ represents the prescription of coefficients behind terms with degree at most $\ell-1$. 

For any ring $\mathcal{R}$, we will use $\mathcal{R}_{\times}$ to denote the group of units in $\mathcal{R}$ and $\mathcal{R}_{+}$ to denote the additive group on $\mathcal{R}$. 

Let $\mathcal{S}$ denote the set $\{x_1^{i_1}\dots x_v^{i_v} \mid i_1+\dots+i_v \le \ell-1\}$. Note that the sets $g \cdot \mathcal{E}_{\times}$ form a partition of $\mathcal{E}$, where $g$ runs over $\mathcal{S}$. For every $c \in \mathcal{E}$, let $S(c)$ denote the unique monomial $g \in \mathcal{E}$ with coefficient $1$ for which $g \mid c$ and $gx_i \nmid c$ for all $1 \le i \le v$, and let $\Red(c)$ represent the polynomial $c/S(c)$ considered over $\mathbb{F}_q[x_1, \dots, x_v]$ and subsequently reduced in $\mathcal{E}$. 
Throughout this paper, we will also use the Iverson bracket $\llbracket P \rrbracket$, which has value $1$ if the predicate $P$ is true and has value $0$ otherwise. 

Next, we set up notation regarding the vanishing sets of polynomials. Given a subset $Z \subseteq \mathbb{F}_q^v \setminus \{\mathbf{0}\}$, let $V(Z)$ denote the set of polynomials in $\mathbb{F}_q[x_1, \dots, x_v]$ that vanish on $Z$ and no other points in $\mathbb{F}_q^v$. Let $N_Z(\varepsilon_0)$ denote the set of polynomials in $V(Z)$ with $x_i$-degree at most $q-1$ for all $1 \le i \le v$ that reduce to $\varepsilon_0$ in $\mathcal{E}$.

Finally, we provide notation and preliminary results on the character sums used in this work. A \emph{character} on a group $G$ is group homomorphism $\phi\colon G \to \{z \in \mathbb{C} : |z| = 1\}$. The set of characters on $G$ is denoted $\widehat{G}$. A \emph{multiplicative character} on a ring $R$ is a character $\chi$ on $R_{\times}$, and an \emph{additive character} on $R$ is a character $\psi$ on $R_+$. The Pontryagin duality theorem implies that $G \cong \widehat{G}$ for all finite abelian groups $G$; in particular, $|G| = |\widehat{G}|$. Given $c \in G$ and $\psi \in \widehat{G}$, let $\psi_c$ denote the additive character defined by the mapping $a \mapsto \psi(ac)$. The \emph{orthogonality of characters} states that for any character $\phi$ defined on a group $G$, $\sum_{g \in G} \phi(g)= |G| \cdot \llbracket \phi(G) = 1 \rrbracket$. Given a group $G$, a subgroup $H$ of $G$ with finite index in $G$, and a character $\phi$ defined on $H$, the number of characters $\overline{\phi}$ on $G$ which are prescribed to $\phi$ on $H$ is $|G:H| = |G|/|H|$. For a multiplicative character $\chi$ and an additive character $\psi$ defined on a ring $\mathcal{R}$, define their \emph{Gauss sum} by \[ G(\chi, \psi) := \sum_{r \in \mathcal{R}_{\times}} \chi(r)\psi(r). \]\noindent With respect to this terminology, our main problem is equivalent to the following: 
\begin{center}
\textsl{Give an estimate on the number of polynomials in $\mathbb{F}_q[x_1, \dots, x_v]$ with all variables $x_i$ having degree at most $q-1$, exactly $s=q^v-k-1$ zeroes, and reduction $\varepsilon_0$ in $\mathcal{E}$. Equivalently, estimate the sum $\sum_{|Z| = s} N_Z(\varepsilon_0)$.} 
\end{center}

\noindent Heuristically, we should expect this quantity to be roughly \[ \frac{1}{|\mathcal{E}|} \sum_{\varepsilon \in \mathcal{E}} \sum_{|Z| = s} N_Z(\varepsilon). \] The key to giving sharp estimates on $\sum_{|Z| = s} N_Z(\varepsilon_0)$ is to give an upper bound on the \emph{error term} \[ \left|\sum_{|Z| = s} N_Z(\varepsilon_0) - \frac{1}{|\mathcal{E}|} \sum_{\varepsilon \in \mathcal{E}} \sum_{|Z| = s} N_Z(\varepsilon)\right|. \] In our main result, we derive an upper bound on the above expression. 

\section{Character Sum Setup}\label{section: char sum setup}
 In this section, we give a character sum formula for $N_Z(\varepsilon_0)$. We first provide an alternative description for $N_Z(\varepsilon_0)$ in Corollary \ref{N_Z description}, based on Lagrange interpolation. Let us begin by proving the following technical fact, which is a standard consequence of the Combinatorial Nullstellensatz~\cite[, Theorem 1.2]{Alon1999}. 

\begin{lemma}\label{combo null application}
    Let $f$ be a polynomial in $\mathcal{M}$ which vanishes on $\mathbb{F}_q^v$. Then, $f$ is the zero polynomial. 
\end{lemma}
\begin{proof}
    Assume for contradiction that $f$ is a nonzero polynomial. Then there exists a nonzero term $cx_1^{t_1}\dots x_v^{t_v}$ in $f(\mathbf{x})$ with $c \in \mathbb{F}_q^{*}$ such that $t_1+\dots+t_v = \deg(f)$. Note that $t_1, \dots, t_v \le q-1$ as $f \in \mathcal{M}$. Consider setting $S_i = \mathbb{F}_q$ for all $1 \le i \le v$. Since $t_i \le q-1 < |S_i|$ for all $1 \le i \le v$, the Combinatorial Nullstellensatz \cite[, Theorem 1.2]{Alon1999} implies that $f(\mathbf{v}) \neq 0$ for some $\mathbf{v} \in S_1 \times \dots \times S_v = \mathbb{F}_q^v$, a contradiction. Hence, $f$ is the zero polynomial.
\end{proof}
An important fact utilized in \cite{LiWan2020} is that the polynomials in $\mathbb{F}_q[x]$ with distinct prescribed zeroes $r_1, \dots, r_k$ can be characterized by the multiples of $(x-r_1)\dots(x-r_k)$. This, infamously, fails to persist in settings with two or more variables, in which polynomials with prescribed zeroes are often characterized by the elements of an ideal generated by two or more elements. The simplest description of such an ideal can be given using Lagrange interpolation; as such, Lagrange interpolation serves as the foundation of our character sum setup, and naturally extends to codes beyond Reed-Muller codes. 
\begin{defn}
    For each $(c_1, \dots, c_v) \in \mathbb{F}_q^v$, define the polynomial \[ P_{(c_1, \dots, c_v)}(\mathbf{x}) := \prod_{i=1}^v(1-(x_i-c_i)^{q-1}). \]
\end{defn}
\begin{lemma}\label{flt-indicator}
For any polynomial $f \in \mathcal{M}$, we have $f \in V(Z)$ if and only if $f$ takes the form $\sum_{\mathbf{v} \in \mathbb{F}_q^v \setminus Z} a_{\mathbf{v}} P_{\mathbf{v}}$ for unique $a_{\mathbf{v}} \in \mathbb{F}_q^{*}$.
\end{lemma}
\begin{proof}
    Let $f$ be a fixed polynomial in $\mathcal{M}$. Observe that for every $\mathbf{v}, \mathbf{w} \in \mathbb{F}_q^v$, we have $P_{\mathbf{v}}(\mathbf{w}) = \llbracket \mathbf{v} = \mathbf{w}\rrbracket$ by Fermat's little theorem. Therefore, if $g$ is the polynomial defined by $\sum_{\mathbf{v} \in \mathbb{F}_q^v \setminus Z} f(\mathbf{v})P_{\mathbf{v}}$, then $f$ and $g$ have the same evaluation at every point in $\mathbb{F}_q^v$. Equivalently, $f-g$ vanishes on $\mathbb{F}_q^v$. Observe that both $f$ and $g$ are in $\mathcal{M}$; thus, $f-g \in \mathcal{M}$. By Lemma \ref{combo null application}, it follows that $f=g$, so we conclude that $f = \sum_{\mathbf{v} \in \mathbb{F}_q^v \setminus Z} f(\mathbf{v})P_{\mathbf{v}}$ for all $f \in \mathcal{M}$. Note that $f$ is in $V(Z)$ if and only if the set of all $\mathbf{v} \in \mathbb{F}_q^v$ for which $f(\mathbf{v}) \in \mathbb{F}_q^{*}$ is exactly $\mathbb{F}_q^v \setminus Z$. Hence, $f \in V(Z)$ if and only if $f$ takes the form $\sum_{\mathbf{v} \in \mathbb{F}_q^v \setminus Z} a_{\mathbf{v}} P_{\mathbf{v}}$ for some $a_{\mathbf{v}} \in \mathbb{F}_q^{*}$. The uniqueness of the $a_{\mathbf{v}}$ follows from Fermat's little theorem.
\end{proof}
As such, we may conclude the following about $N_Z(\varepsilon_0)$. 
\begin{cor}\label{N_Z description}
    The number of tuples $(a_{\mathbf{v}})_{\mathbf{v} \in \mathbb{F}_q^v \setminus Z}$ of elements of $\mathbb{F}_q^{*}$ for which $\sum_{\mathbf{v} \in \mathbb{F}_q^v \setminus Z} a_{\mathbf{v}}P_{\mathbf{v}}$ reduces to $\varepsilon_0$ in $\mathcal{E}$.
\end{cor}

We now give a character sum expression for $N_Z(\varepsilon_0)$. This results in an error bound for $\sum_{|Z| = s} N_Z(\varepsilon_0)$.

\begin{prop}\label{char sum formula for N_Z}
    For all $\alpha \in \mathbb{F}_q^v \setminus \{\mathbf{0}\}$, $\chi \in \widehat{\mathcal{E}_{\times}}$, and $\psi \in \widehat{\mathcal{E}_{+}}$, set 
    \[ \mathbf{Q}_{\alpha}(\chi, \psi) := \frac{q}{|\mathcal{E}|} \chi(\Red(P_{\alpha}))G(\chi^{-1}, \psi_{S(P_{\alpha})}).\] Then we have \[ \left|\sum_{|Z| = s} N_Z(\varepsilon_0) - \frac{1}{|\mathcal{E}|}\sum_{\varepsilon \in \mathcal{E}} \sum_{|Z| = s} N_Z(\varepsilon)\right| \le \frac{1}{|\mathcal{E}|} \sum_{\substack{\psi \in \widehat{\mathcal{E}_{+}} \\ \psi \neq \mathbf{1}}} \left|\sum_{|Z| = s} \prod_{\alpha \in \mathbb{F}_q^v \setminus Z} \left(\sum_{\substack{\chi \in \widehat{\mathcal{E}_{\times}} \\ \chi(\mathbb{F}_q^{*}) = 1}} \mathbf{Q}_{\alpha}(\chi, \psi)\right)\right|. \]
\end{prop}
\begin{proof}
   Observe that for any two polynomials $f_1, f_2 \in \mathbb{F}_q[x_1, \dots, x_v]$, we have $f_1 = f_2$ if and only if $\Red(f_1)= \Red(f_2)$ and $S(f_1)=S(f_2)$. Thus, the Iverson bracket $\llbracket f_1 = f_2 \rrbracket$ is given by \[ \llbracket f_1 = f_2 \rrbracket = \frac{1}{|\mathcal{E}_{\times}|} \sum_{\substack{\chi \in \widehat{\mathcal{E}_{\times}}}} \chi(\Red(f_1))\chi(\Red(f_2))^{-1}\llbracket S(f_1)=S(f_2)\rrbracket \] by the orthogonality of multiplicative characters. For all $\alpha \in \mathbb{F}_q^v \setminus \{\mathbf{0}\}$, $\chi \in \widehat{\mathcal{E}_{\times}}$, and $\psi \in \widehat{\mathcal{E}_{+}}$, set \begin{align*} \mathbf{P}_{\alpha}(\chi, \psi) &:= \frac{1}{|\mathcal{E}_{\times}|}\sum_{\substack{u \in \mathcal{E}_{\times} \\ a \in \mathbb{F}_q^{*}}} \chi(a\Red(P_{\alpha}))\chi(u)^{-1}\psi(S(P_{\alpha})u) \\ &= \llbracket \chi(\mathbb{F}_q^{*}) = 1 \rrbracket \cdot \frac{q-1}{|\mathcal{E}_{\times}|} \sum_{u \in \mathcal{E}_{\times}} \chi(\Red(P_{\alpha})) \chi(u)^{-1}\psi(S(P_{\alpha})u) \\ &= \llbracket \chi(\mathbb{F}_q^{*}) = 1 \rrbracket \cdot \frac{q}{|\mathcal{E}|}\chi(\Red(P_{\alpha}))G(\chi^{-1}, \psi_{S(P_{\alpha})}). \end{align*} By the orthogonality of additive characters, it follows that 
    \begin{align*} N_Z(\varepsilon_0) &= \frac{1}{|\mathcal{E}|}
    \sum_{\psi \in \widehat{\mathcal{E}_{+}}} \psi(\varepsilon_0)^{-1} \prod_{\alpha \in \mathbb{F}_q^v \setminus Z} \left(\sum_{\substack{\chi \in \widehat{\mathcal{E}_{\times}} \\ \chi(\mathbb{F}_q^{*}) = 1}} \mathbf{P}_{\alpha}(\chi, \psi)\right). \end{align*} Observe that 
     \begin{align*} \sum_{\varepsilon \in \mathcal{E}} N_Z(\varepsilon) &= \frac{1}{|\mathcal{E}|}
    \sum_{\substack{\psi \in \widehat{\mathcal{E}_{+}} \\ \psi(\mathcal{E}_{+}) = 1}} \psi(\varepsilon_0)^{-1}\prod_{\alpha \in \mathbb{F}_q^v \setminus Z} \left(\sum_{\substack{\chi \in \widehat{\mathcal{E}_{\times}} \\ \chi(\mathbb{F}_q^{*}) = 1}} \mathbf{P}_{\alpha}(\chi, \psi)\right).\end{align*} Thus, the desired error term $E_Z(d, \varepsilon_0)$ is given by 
    \begin{align*} E_Z(d, \varepsilon_0) &= \frac{1}{|\mathcal{E}|}
    \sum_{\substack{\psi \in \widehat{\mathcal{E}_{+}} \\ \psi \neq \mathbf{1}}} \psi(\varepsilon_0)^{-1} \prod_{\alpha \in \mathbb{F}_q^v \setminus Z} \left(\sum_{\substack{\chi \in \widehat{\mathcal{E}_{\times}} \\ \chi(\mathbb{F}_q^{*}) = 1}} \mathbf{P}_{\alpha}(\chi, \psi)\right). \end{align*} 
    By the triangle inequality, we have \[ \left|\sum_{|Z|=s} E_Z(d, \varepsilon_0)\right| \le \frac{1}{|\mathcal{E}|} \sum_{\substack{\psi \in \widehat{\mathcal{E}_{+}} \\ \psi \neq \mathbf{1}}} \left|\sum_{|Z| = s} \prod_{\alpha \in \mathbb{F}_q^v \setminus Z} \left(\sum_{\substack{\chi \in \widehat{\mathcal{E}_{\times}} \\ \chi(\mathbb{F}_q^{*}) = 1}} \mathbf{P}_{\alpha}(\chi, \psi)\right)\right|, \] which completes the proof.  
\end{proof}

\section{Center of approximation}\label{section: center of approximation}
Recall that the center of approximation we use for $N_Z(\varepsilon_0)$ is given by \[ \mathbb{E}_{\varepsilon_0, Z}[N_Z(\varepsilon_0)] = \frac{1}{|\mathcal{E}|} \sum_{\varepsilon_0 \in \mathcal{E}} \sum_{|Z| = s} N_Z( \varepsilon_0), \] which can be described as $1/(|\mathcal{E}|)$ multiplied by the size of the set $M_v(s)$ of polynomials $f$ in $\mathbb{F}_q[x_1, \dots, x_v]$ such that the degree of $x_i$ in $f$ is at most $q-1$ for all $1 \le i \le v$, $f$ has a nonzero constant term, and $f$ has exactly $s$ zeroes. This type of enumeration, though without the nonzero constant term condition, has been studied previously; in particular, see \cite[, Theorem 4.5]{Jain_2025}. We will adjust the proof of \cite[, Theorem 4.5]{Jain_2025} to take into account of the additional constraint of $f$ having a nonzero constant term. Let $\Phi_{v}(t)$ denote the generating function in $t$ for which the coefficient of $t^k$ in $\Phi_{v}(t)$ represents the probability that a random polynomial $f \in \bigsqcup_{s \ge 0} M_v(s)$ satisfies $f \in M_v(k)$. 
\begin{claim}\label{phi_v and phi_1}
    We have $\Phi_{v}(t) = (\Phi_{1}(t))^{q^{v-1}}$. 
\end{claim}
\begin{proof}
    It is not difficult to check that the part of the proof of \cite[, Theorem 4.5]{Jain_2025} which proves this fact remains valid in this context as well. 
\end{proof}
Thus, it suffices to compute $\Phi_1(t)$. 
\begin{lemma}\label{phi_1 closed form}
    The generating function $\Phi_1(t)$ is given by \[ \Phi_1(t) =  \left(1+\frac{t-1}{q}\right)^{q-1}. \]
\end{lemma}
\begin{proof}
    Let $r_1, \dots, r_k$ be distinct elements of $\mathbb{F}_q^{*}$. Note that the number of polynomials $f(x) \in \mathbb{F}_q[x]$ of degree at most $q-1$ and nonzero constant term which vanish at $r_1, \dots, r_k$ is $(q-1)q^{q-k-1}$. Thus, by the inclusion-exclusion principle, the number of polynomials $f(x) \in \mathbb{F}_q[x]$ of degree at most $q-1$, nonzero constant term, and exactly $k$ roots is given by \begin{align*} \sum_{j=k}^{q-1} \binom{q-1}{j} \cdot (q-1)q^{q-j-1}(-1)^j &= (q-1)q^{q-1}\sum_{j=k}^{q-1} \binom{q-1}{j}(-q)^{-j} \\ &= (q-1)q^{q-1} \cdot [t^k]\left(1+\frac{t-1}{q}\right)^{q-1}. \end{align*} Hence we have \[ \Phi_1(t) = \left(1+\frac{t-1}{q}\right)^{q-1}, \] as desired.
\end{proof}
Combining Claim \ref{phi_v and phi_1} and Lemma \ref{phi_1 closed form}, we have \[ \Phi_v(t) = \left(1+\frac{t-1}{q}\right)^{(q-1)q^{v-1}}. \] Upon expansion, we compute \[ |M_v(s)| = \binom{(q-1)q^{v-1}}{s} \cdot q^{(q-1)q^{v-1} - s}, \] from which the closed form for our center of approximation is \begin{equation} \mathbb{E}_{\varepsilon_0, Z}[N_Z(\varepsilon_0)] = \frac{|M_v(s)|}{|\mathcal{E}|} = \binom{(q-1)q^{v-1}}{s}\frac{q^{(q-1)q^{v-1}-s}}{|\mathcal{E}|}. \end{equation}

\section{Gauss Sum Computation}\label{section: gauss sums}
For Sections \ref{section: gauss sums} and \ref{section: key enumerations}, denote $\mathcal{R}$ by the ring \[ \mathbb{F}_q[x_1, \dots, x_v]/\langle x_1^{i_1}\dots x_v^{i_v} \mid i_1+\dots+i_v = n\rangle. \] In this section, we compute the sizes of Gauss sums over $\mathcal{R}$. \\

Let $\chi$ and $\psi$ be a non-trivial multiplicative character and non-trivial additive character on $\mathcal{R}$, respectively. We have \[ G(\chi, \psi)\overline{G}(\chi, \psi) = \sum_{a, b \in \mathcal{R}_{\times}} \chi(a)\psi(a)\overline{\chi}(b)\overline{\psi}(b) = \sum_{a, b \in \mathcal{R}_{\times}} \chi(ab^{-1})\psi(a-b). \] Thus, by performing casework on $ab^{-1}$, we obtain \[ G(\chi, \psi)\overline{G}(\chi, \psi) = \sum_{c \in \mathcal{R}_{\times}} \chi(c) \left(\sum_{d \in \mathcal{R}_{\times}}\psi(d(c-1))\right). \] When $c(0, \dots, 0) \neq 1$, we have \[ \sum_{d \in \mathcal{R}_{\times}}\psi(d(c-1)) = 0 \] as $\psi \neq \mathbf{1}$.  Note that \begin{align*} (p-1)G(\chi, \psi)\overline{G}(\chi, \psi) &= \sum_{c \in \mathcal{R}_{\times}} \chi(c) \left(\sum_{d \in \mathcal{R}_{\times}} \sum_{j=1}^{p-1} \psi(jd(c-1))\right) \\ &= \sum_{c \in \mathcal{R}_{\times}} \chi(c) \left(\sum_{d \in \mathcal{R}_{\times}} \sum_{j=1}^{p-1} \psi(d(c-1))^j \right) \\ &= p\sum_{c \in \mathcal{R}_{\times}} \chi(c)\left(\sum_{d \in \mathcal{R}_{\times}} \llbracket \psi(d(c-1)) = 1 \rrbracket \right) - \sum_{c \in \mathcal{R}_{\times}} \chi(c) \left(\sum_{d \in \mathcal{R}_{\times}} 1\right) \\ &=  p\sum_{c \in \mathcal{R}_{\times}} \chi(c)\left(\sum_{d \in \mathcal{R}_{\times}}\llbracket \psi(d(c-1)) = 1 \rrbracket \right). \end{align*}
\begin{defn}
For all $c \in \mathcal{R}$, set $g(\psi, c) := \sum_{d \in \mathcal{R}_{\times}} \llbracket \psi(dc) = 1 \rrbracket$. 
\end{defn}
\begin{lemma}\label{kernel cap G_c size}
    We have \[ g(\psi, c-1) = \frac{(p-1)|\mathcal{R}|}{p} \cdot \llbracket \psi((c-1)\mathcal{R}) = 1 \rrbracket. \]
\end{lemma}
\begin{proof}
    First, we claim that \[ \sum_{d \in \mathcal{R}} \llbracket \psi(dc) = 1 \rrbracket = \frac{|\mathcal{R}|}{p} \cdot p^{\llbracket \psi((c-1)\mathcal{R}) = 1\rrbracket}. \] Suppose that $\psi((c-1)\mathcal{R}) \neq 1$, as $\ker(\psi) \cap (c-1)\mathcal{R} = (c-1)\mathcal{R}$ when $\psi((c-1)\mathcal{R}) = 1$. Let $\{u_1, \dots, u_{\log_p(q)}\}$ be an $\mathbb{F}_p$-basis for $\mathbb{F}_q$. Note that every element of $(c-1)\mathcal{R}$ can be uniquely expressed as a linear combination of monomials of the form $u_k \cdot M$, where $M$ is a monomial with coefficient $1$. Furthermore, note that there is a monomial $u_k \cdot M$ such that $\psi(u_k \cdot M) \neq 1$. For any linear combination of all monomials except for $u_k \cdot M$, there is a unique choice of $1 \le j \le p-1$ for which a coefficient of $j$ behind $u_k \cdot M$ will result in the polynomial lying in $\ker(\psi)$. This implies that $|\ker(\psi) \cap (c-1)\mathcal{R}| = \frac{1}{p} \cdot |(c-1)\mathcal{R}|$, as claimed.

    Now we will show that $g(\psi, c-1) = \frac{(p-1)|\mathcal{R}|}{p} \cdot \llbracket \psi((c-1)\mathcal{R}) = 1 \rrbracket$. By the inclusion-exclusion principle, we have \[ g(\psi, c-1) = \sum_{T \subseteq [v]} \frac{|\mathcal{R}|}{p} \cdot p^{\llbracket \psi((c-1)x_T\mathcal{R}) = 1 \rrbracket} (-1)^{|T|}. \] We can rewrite the above as \begin{align*} \sum_{T \subseteq [v]} \frac{|\mathcal{R}|}{p} \cdot p^{\llbracket \psi((c-1)x_T\mathcal{R}) = 1 \rrbracket} (-1)^{|T|}  &= \frac{|\mathcal{R}|}{p} \cdot \left(\sum_{T \subseteq [v]} (p-1)(-1)^{|T|}\llbracket \psi((c-1)x_T\mathcal{R}) = 1 \rrbracket \right) + \frac{|\mathcal{R}|}{p} \cdot \left(\sum_{T \subseteq [v]} (-1)^{|T|}\right) \\ &= \frac{(p-1)|\mathcal{R}|}{p} \cdot \sum_{T \subseteq [v]} (-1)^{|T|}\llbracket \psi((c-1)x_T\mathcal{R}) = 1 \rrbracket. \end{align*} By another application of inclusion-exclusion, we compute \[ \sum_{T \subseteq [v]} (-1)^{|T|}\llbracket \psi((c-1)x_T\mathcal{R}) = 1 \rrbracket = \llbracket \psi((c-1)\mathcal{R}_{\times}) = 1 \rrbracket. \] Note that $\psi((c-1)\mathcal{R}_{\times}) = 1$ holds if and only if $\psi((c-1)\mathcal{R}) = 1$, as $\mathcal{R}_{\times} - \mathcal{R}_{\times} = \mathcal{R}$. Hence, we conclude that \[ g(\psi, c-1) = \frac{(p-1)|\mathcal{R}|}{p} \cdot \llbracket \psi((c-1)\mathcal{R}) = 1 \rrbracket, \] as desired. \end{proof}
Since $g(\psi, c)$ is determined by whether or not $c\mathcal{R}$ is contained in $\ker(\psi)$, it is important to characterize all $c$ for which $\psi(c\mathcal{R})=1$; we do so later in Lemma \ref{K_psi is an ideal}. 
\begin{defn}
We define $K_{\psi}$ by the set of all $c \in \mathcal{R}$ such that $\psi(c\mathcal{R}) = 1$.
\end{defn}

The main result of this section is the following. 

\begin{prop}\label{gauss sum size formula}
    For non-trivial characters $\chi$ and $\psi$, we have \begin{align*} |G(\chi, \psi)|^2 &= |\mathcal{R}||K_{\psi}| \cdot \llbracket \chi(1+K_{\psi}) = 1 \rrbracket. \end{align*}
\end{prop}
\begin{proof}
    Merging our computations thus far, we have \begin{align*} |G(\chi, \psi)|^2 &= \frac{p}{p-1}\sum_{c \in \mathcal{R}_{\times}} \chi(c)g(\psi, c-1) \\ &= \frac{p}{p-1} \cdot \frac{(p-1)|\mathcal{R}|}{p} \sum_{c \in \mathcal{R}_{\times}} \chi(c) \cdot \llbracket \psi((c-1)\mathcal{R}) = 1 \rrbracket \\ &= |\mathcal{R}|\sum_{\substack{c \in \mathcal{R}_{\times} \\ c \in 1+K_{\psi}}} \chi(c) \\ &= |\mathcal{R}||K_{\psi}| \cdot \llbracket \chi(1+K_{\psi}) = 1 \rrbracket. \end{align*}
\end{proof}

\section{Key Enumerations}\label{section: key enumerations}
\subsection{Distribution of codimension of ideals}\label{subsection: enumeration of ideals}
\noindent In this subsection, we will give all necessary enumerations that pertain to the distribution of $\operatorname{codim}(K_{\psi})$ over additive characters $\psi$ on $\mathcal{R}$. We begin by showing that, in fact, $K_{\psi}$ is always an ideal. 

\begin{lemma}\label{K_psi is an ideal}
    Let $\psi$ be an additive character on $\mathcal{R}$. The set $K_{\psi}$ is the unique ideal of $\mathcal{R}$ that is maximally contained in $\ker(\psi)$.  
\end{lemma}
\begin{proof}
    Note that $K_{\psi}$ is an additive subgroup of $\mathcal{R}$ and that for all $c \in K_{\psi}$ and $c' \in \mathcal{R}$, we have $\psi(cc'\mathcal{R}) = \psi(c\mathcal{R}) = 1$. Thus, $K_{\psi}$ is an ideal of $\mathcal{R}$. Observe that $K_{\psi}$ contains all ideals of $\mathcal{R}$ contained in $\ker(\psi)$; otherwise, there exists an element $c \not\in K_{\psi}$ with $c\mathcal{R}$ contained in $\ker(\psi)$, a contradiction. 
\end{proof}

Now, given an ideal $I$, we count the characters $\psi$ for which $K_{\psi}=I$ in Lemma \ref{counting additive chars}. This result, alongside providing counts, gives us the space of ideals we need to consider. 

\begin{defn}
  Define the \emph{socle} $\operatorname{Soc}(\mathcal{R}/I)$ of the quotient ring $\mathcal{R}/I$ by the additive subgroup of $\mathcal{R}/I$ given by $\{f \in \mathcal{R}/I \mid x_1f = 0, \dots, x_vf =0\}$. 
\end{defn}
\begin{lemma}\label{counting additive chars}
    Let $I$ be an ideal of $\mathcal{R}$. Then, the number of additive characters $\psi$ on $\mathcal{R}$ such that $I$ is an ideal maximally contained in $\ker(\psi)$ is \[ \frac{|\mathcal{R}|}{|I|} \cdot \begin{cases} 1 & \dim(\operatorname{Soc}(\mathcal{R}/I))=0 \\ 1-\frac{1}{q} & \dim(\operatorname{Soc}(\mathcal{R}/I)) = 1 \\ 0 & \text{otherwise.} \end{cases}\]
\end{lemma}
\begin{proof}
   It is well-known that the Möbius function $\mu(I, J)$ on the lattice of ideals of $\mathcal{R}$ is given by $(-1)^rq^{\binom{r}{2}}$, where $\dim(J) - \dim(I) = r$ (see, for example, \cite[eq.\ (3.34)]{StanleyECV1}). Thus, by Möbius inversion, the desired number of additive characters becomes \[ \sum_{\substack{I \subseteq J \subseteq \mathcal{R} \\ \dim(J)-\dim(I) = r}} (-1)^{r} q^{\binom{r}{2}} \cdot \frac{|\mathcal{R}|}{|J|} = \frac{|\mathcal{R}|}{|I|} \sum_{\substack{I \subseteq J \subseteq \mathcal{R} \\ \dim(J)-\dim(I) = r}} (-1)^{r} q^{\binom{r}{2}-r}, \] where all sums above run over ideals $J$. Re-parametrizing the above as a sum over all subspaces $J/I$ of $\operatorname{Soc}(\mathcal{R}/I)$, we may rewrite the above sum as \[ \frac{|\mathcal{R}|}{|I|}\sum_{A \subseteq \operatorname{Soc}(\mathcal{R}/I)} (-1)^{\dim(A)}q^{\binom{\dim(A)}{2}-\dim(A)}. \] Since there are ${\dim(\operatorname{Soc}(\mathcal{R}/I)) \brack r}_q$ subspaces of $\dim(\operatorname{Soc}(\mathcal{R}/I))$ with dimension $r$, it follows that the above sum is \[ \frac{|\mathcal{R}|}{|I|}\sum_{r=0}^{\dim(\operatorname{Soc}(\mathcal{R}/I))} {\dim(\operatorname{Soc}(\mathcal{R}/I)) \brack r}_q(-1)^{r}q^{\binom{r}{2}-r} = \frac{|\mathcal{R}|}{|I|}\prod_{k=0}^{\dim(\operatorname{Soc}(\mathcal{R}/I))-1} (1-q^{k-1}) \] by the $q$-binomial theorem. Thus, the desired result follows. 
\end{proof}
Motivated by Lemma \ref{counting additive chars}, we characterize the ideals $I$ of $\mathcal{R}$ with $\dim(\operatorname{Soc}(\mathcal{R}/I)) = 1$. 

\begin{defn}
  For all $0 \le k \le n-1$, let $V(n)$ denote the subspace of $\mathcal{R}$ consisting of $0$ and all degree-$n$ homogeneous polynomials. Furthermore, for each ideal $I$ of $\mathcal{R}$, we denote $I(n) = I \cap V(n)$. 
\end{defn}

\begin{prop}\label{characterization of soc-1 ideals}
  Let $I$ be an ideal of $\mathcal{R}$. Then $\dim(\operatorname{Soc}(\mathcal{R}/I)) = 1$ if and only if there is a positive integer $1 \le n_0 \le n-1$ such that
  \begin{itemize}
    \item $\dim(I(n_0)) = \dim(V(n_0)) - 1$;
    \item $I(k) = V(k)$ for all $k > n_0$;
    \item $I(k) = \{f \in V(k) \mid fx_1, \dots, fx_v \in I(k+1)\}$ for all $k < n_0$;
    \item $I$ is homogeneous (that is, $I = \bigoplus_{k=0}^{n-1} I(k)$).
  \end{itemize}
\end{prop}
\begin{proof}
  It is straightforward to check that ideals $I$ satisfying the four given properties also satisfy $\dim(\operatorname{Soc}(\mathcal{R}/I)) = 1$. Let $n_0$ denote the largest positive integer for which $I(n_0) \neq V(n_0)$; we will show that this $n_0$ satisfies the hypotheses of the lemma.

  If $\dim(I(n_0)) < \dim(V(n_0)) - 1$, then there exist two elements $f, g \in V(n_0)$ such that $f \cdot \mathbb{F}_q \neq g \cdot \mathbb{F}_q$, $x_if \in I$, and $x_ig \in I$ for all $1 \le i \le v$. This implies that $\dim(\operatorname{Soc}(\mathcal{R}/I)) \ge 2$, so we must have $\dim(I(n_0)) = \dim(V(n_0)) - 1$. Thus, the first two of the four properties hold. We also derive that for all $k < n_0$, we have $\{f \in V(k) \mid fx_1, \dots, fx_v \in I(k+1)\} \subseteq I(k)$. Since $I$ is an ideal, $I(k) \subseteq \{f \in V(k) \mid fx_1, \dots, fx_v \in I(k+1)\}$; thus, the third property holds. 

  Now, we prove that $I$ is homogeneous. Assume for contradiction that there exists an element $f \in I$ such that if $f_k$ denotes the degree-$k$ component of $f$ for all $k$, then there is an index $j$ for which $f_j \not\in V(j)$. Let $j_0$ denote the minimal index for which $f_{j_0} \not\in I(j_0)$. By the third property, we have $f_{j_0} \cdot x_1^{n_0-j_0} \not\in I$. Observe that \[\left(\sum_{j \ge j_0} f_j\right) \cdot x_1^{n_0-j_0} = f_{j_0} \cdot x_1^{n_0-j_0} + F, \] where $F \in V(n_0+1) + \dots + V(n-1) \subseteq I$. Thus, $\left(\sum_{j \ge j_0} f_j\right) \cdot x_1^{n_0-j_0} \not\in I$. Since $I$ is an ideal, we have $\sum_{j \ge j_0} f_j \not\in I$. Our selection of $j_0$ establishes that $\sum_{j < j_0} f_j \in I$, so in summary, \[ f = \left(\sum_{j < j_0} f_j\right) + \left(\sum_{j \ge j_0} f_j\right) \not\in I, \] a contradiction. Therefore, $I$ is homogeneous, which completes the proof. 
\end{proof}
With Proposition \ref{characterization of soc-1 ideals} established, we denote $n_0(I)$ by the largest positive integer $n_0 \le n-1$ such that $\dim(I(n_0)) < \dim(V(n_0))$. \\

The characterization in Proposition \ref{characterization of soc-1 ideals} is sufficient for us to determine upper bounds on the distribution of $\operatorname{codim}(I)$ over all ideals $I$ of $\mathcal{R}$ satisfying $\dim(\operatorname{Soc}(\mathcal{R}/I))=1$. As such, we may biject these ideals to generalized Hankel matrices with prescribed ranks. 

\begin{defn}
  Let $z$ be an indeterminate. Define the \emph{codimension enumerator} $F_n(z)$ of $\mathcal{R}$ by the generating function \[ F_n(z) := \sum_{\dim(\operatorname{Soc}(\mathcal{R}/I)) = 1} z^{\operatorname{codim}(I)}. \]
\end{defn}

\begin{defn}
  Let $S=(a_{(i_1, \dots, i_v)})_{i_1+\dots+i_v = m}$, where the indices $i_1, \dots, i_v$ are nonnegative integers, be a multidimensional sequence with entries in $\mathbb{F}_q$. Define the \emph{$j$-Hankel matrix} $M(S, j)$ on $S$ by the $\binom{(m-j)+v-1}{v-1} \times \binom{j+v-1}{v-1}$ matrix associated with the map \[ \varphi: \{(s_1, \dots, s_v) \mid s_1+\dots+s_v = m-j\} \times \{(t_1, \dots, t_v) \mid t_1+\dots+t_v=j\} \to \mathbb{F}_q \] given by \[ \varphi((s_1, \dots, s_v), (t_1, \dots, t_v)) = a_{(s_1+t_1, \dots, s_v+t_v)}. \] 
\end{defn}
\begin{rem*}
  The $v=2$ case represents regular Hankel matrices. 
\end{rem*}
\begin{prop}\label{ideals and LC bijection}
  For all positive integers $n$ and real numbers $z > 1$, we have \[ F_n(z) \le z^{2\binom{\lceil n/2 \rceil+v}{v}}q^{\binom{n+v-2}{v-1}} \cdot \frac{q}{(q-1)^2}. \]
\end{prop}
\begin{proof}
  Consider an ideal $I_0$ of $\mathcal{R}$ with $\dim(\operatorname{Soc}(\mathcal{R}/I_0)) = 1$, and let $n_0 := n_0(I_0)$. Then, $\dim(I_0(n_0)) = \dim(V(n_0)) - 1$. Observe that, up to scaling, there exists a unique sequence $S=(c_{(i_1, \dots, i_v)})_{i_1+\dots+i_v=n_0}$ of constants in $\mathbb{F}_q$, not all $0$, such that \[ I_0(n_0) = \left\{\sum_{i_1+\dots+i_v=n_0} a_{(i_1, \dots, i_v)}x_1^{i_1}\dots x_v^{i_v} \,\middle\vert\, \sum_{i_1+\dots+i_v=n_0} a_{(i_1, \dots, i_v)}c_{(i_1, \dots, i_v)} = 0\right\}. \] By the third property of Proposition \ref{characterization of soc-1 ideals}, it follows that \[ I_0(k) = \left\{\sum_{i_1+\dots+i_v=n_0} a_{(i_1, \dots, i_v)}x_1^{i_1}\dots x_v^{i_v} \,\middle\vert\, (a_{(i_1, \dots, i_v)})_{i_1+\dots+i_v=n_0} \in \ker(M(S, k))  \right\} \] for all $1 \le k \le n_0$. Thus, for all $1 \le k \le n_0$, we have $\operatorname{codim}(I_0(k)) = \operatorname{rank}(M(S, k))$. This means that \[ \operatorname{codim}(I_0) = \sum_{k=0}^{n_0} \operatorname{codim}(I_0(k)) = \sum_{k=0}^{n_0} \operatorname{rank}(M(S, k)), \] as $I_0$ is homogeneous. Now, note the upper bound \[ \operatorname{rank}(M(S, k)) \le \binom{\min(k, n_0-1-k)+v-1}{v-1}. \] Thus, \begin{align*} \operatorname{codim}(I) &\le \sum_{k=0}^{n-1} \binom{\min(k, n-1-k)+v-1}{v-1} \\ &= \sum_{k=0}^{\lfloor n/2 \rfloor - 1} \binom{k+v-1}{v-1} + \sum_{k=\lfloor n/2 \rfloor}^{n-1} \binom{(n-1-k)+v-1}{v-1} \\ &= \sum_{k=0}^{\lfloor n/2 \rfloor-1} \binom{k+v-1}{v-1} + \sum_{k=0}^{\lceil n/2 \rceil - 1} \binom{k+v-1}{v-1} \\ &\le 2\binom{\lceil n/2 \rceil + v-1}{v} \end{align*} for all ideals $I$ of $\mathcal{R}$ with $\dim(\operatorname{Soc}(\mathcal{R}/I)) = 1$. Since there are a total of \begin{align*} \sum_{n_0=1}^{n-1} {\dim(V(n_0)) \brack 1}_q &\le \sum_{i=0}^{\infty} \frac{1}{q^i} \cdot {\dim(V(n-1)) \brack 1}_q \\ &\le \sum_{i=0}^{\infty} \frac{1}{q^i} \cdot \frac{1}{q-1} \cdot q^{\dim(V(n-1))} \\ &=  \frac{q}{(q-1)^2} \cdot q^{\binom{n+v-2}{v-1}} \end{align*} ideals $I$ of $\mathcal{R}$ with $\dim(\operatorname{Soc}(\mathcal{R}/I)) = 1$, the codimension enumerator of $\mathcal{R}$ is at most \[ z^{2\binom{\lceil n/2 \rceil+v}{v}}q^{\binom{n+v-2}{v-1}} \cdot \frac{q}{(q-1)^2}, \] as desired. 
\end{proof}
Proposition \ref{ideals and LC bijection} successfully completes the task of finding an upper bound on the distribution of codimension over all ideals $I$ that satisfy $\dim(\operatorname{Soc}(\mathcal{R}/I))=1$. In Section \ref{section: li-wan sieve}, we also consider ideals of the form $K_{\psi_{x_T}}$, where $T$ is a subset of $[v]$. Thus, we require the following.
\begin{prop}
  Let $T$ be a subset of $[v]$. Then \[ \sum_{\dim(\operatorname{Soc}(\mathcal{R}/I))=1} z^{\operatorname{codim}(I \cap x_T\mathcal{R})} \le \frac{q}{(q-1)^2} \cdot q^{\binom{n+v-2}{v-1}-\binom{n-|T|+v-2}{v-1}} + F_{n-|T|}(z). \]  
\end{prop}
\begin{proof}
  The number of ideals $I$ of $\mathcal{R}$ with $\dim(\operatorname{Soc}(\mathcal{R}/I))=1$ that contain $x_T\mathcal{R}$ is \begin{align*} \sum_{i=1}^{n-1} {\binom{i+v-1}{v-1}-\binom{\max(0, i-|T|)+v-1}{v-1} \brack 1}_q &= \frac{1}{q-1} \sum_{i=1}^{n-1} (q^{\binom{i+v-1}{v-1}-\binom{\max(0, i-|T|)+v-1}{v-1}}-1) \\ &\le \frac{q}{(q-1)^2} \cdot q^{\binom{n+v-2}{v-1}-\binom{n-|T|+v-2}{v-1}},  \end{align*} by casework on $n_0(I)$ and the first property of Proposition \ref{characterization of soc-1 ideals}. If $I$ does not contain $x_T\mathcal{R}$, then $\dim(\operatorname{Soc}(x_T\mathcal{R}/(I \cap x_T\mathcal{R})))=1$; thus, such $I$ return a contribution of $F_{n-|T|}(z)$. In total, we have \[ \sum_{\dim(\operatorname{Soc}(\mathcal{R}/I))=1} z^{\operatorname{codim}(I \cap x_T\mathcal{R})} \le \frac{q}{(q-1)^2} \cdot q^{\binom{n+v-2}{v-1}-\binom{n-|T|+v-2}{v-1}} + F_{n-|T|}(z), \] as desired. 
\end{proof}

\subsection{Enumerating multiplicative characters}\label{subsection: counting mult chars} We will now prove results on multiplicative character counts. An important idea that proves useful in Section \ref{section: li-wan sieve} is that multiplicative characters and their products tend to have a ``uniform distribution.'' Taking advantage of this fact, which we present probabilistically in Proposition \ref{mult chars probability}, can lead to convenient simplifications in sums over multiplicative characters. 
\begin{lemma}\label{mult chars counting lemma}
    Let $G$ be a finite abelian group and let $G_1, \dots, G_n, H$ be subgroups of $G$. Fix a character $\chi$ on $G$. The number of tuples of characters $(\chi_1, \dots, \chi_n) \in (\widehat{G})^n$ such that 
    \begin{itemize} 
    \item $\chi_i(G_i) = 1$ for every $1 \le i \le n$, and
    \item $(\prod_{i=1}^n \chi_i)(h) = \chi(h)$ for all $h \in H$
\end{itemize}
is given by \[ \frac{|H \cap K|}{|H|} \cdot \prod_{i=1}^n \frac{|G|}{|G_i|} \cdot \llbracket \chi(H \cap K) = 1 \rrbracket, \] where $K := \bigcap_{i=1}^n G_i$. 
\end{lemma}
\begin{proof}
   Let $\Delta$ denote the desired count. Note that \begin{align*} \Delta \cdot \left(|H| \prod_{i=1}^n |G_i|\right) &= \sum_{\chi_i \in \widehat{G}} \sum_{g_i \in G_i} \sum_{h \in H} \chi_1(g_1)\cdots\chi_n(g_n)\left(\prod_{i=1}^n \chi_i\right)(h)\chi^{-1}(h) \\ &= \sum_{\chi_i \in \widehat{G}} \sum_{g_i \in G_i} \sum_{h \in H} \chi_1(g_1h)\cdots\chi_n(g_nh)\chi^{-1}(h) \end{align*} by the orthogonality of characters. Thus, \begin{align*} \Delta \cdot \left(|H| \prod_{i=1}^n |G_i|\right) &=  |G|^n \sum_{g_i \in G_i} \sum_{h \in H} \llbracket g_1h = 1 \rrbracket \cdots \llbracket g_nh = 1 \rrbracket \chi^{-1}(h) \\ &= |G|^n \sum_{h \in H} \llbracket h \in K \rrbracket \chi^{-1}(h) \\ &= |G|^n \cdot  |H \cap K| \cdot \llbracket \chi^{-1}(H \cap K) = 1 \rrbracket \\ &= |G|^n \cdot  |H \cap K| \cdot \llbracket \chi(H \cap K) = 1 \rrbracket,  \end{align*} so we have $\Delta = \frac{|H \cap K|}{|H|} \cdot  \prod_{i=1}^n \frac{|G|}{|G_i|} \cdot \llbracket \chi(H \cap K) = 1 \rrbracket$, as desired. 
\end{proof}
In particular, we have the following.
\begin{prop}\label{mult chars probability}
    Let $I$ be an ideal of $\mathcal{R}$ with $\dim(\operatorname{Soc}(\mathcal{R}/I)) = 1$, and let $T_0$ be a nonempty subset of $[v]$. Fix a character $\chi \in \widehat{\mathcal{R}_{\times}}$. The probability that a sequence $(\chi_i)_{i=1}^k$ of multiplicative characters on $\mathcal{R}$ with $\chi_i(1+I)=1$ and $\chi_i(\mathbb{F}_q^{*})=1$ for all $1 \le i \le k$ chosen uniformly at random satisfies \[ \left(\prod_{i=1}^k \chi_i\right)(x_l-a) = \chi(x_l-a) \] for all $a \in \mathbb{F}_q^{*}$ and $l \in T_0$ is at most \[ \left(\sum_{k=1}^{\lfloor p/2 \rfloor} \binom{q}{k} \cdot \left(\frac{n}{p}\right)^k\right)^{-|T_0|}. \]
\end{prop}
\begin{proof}
  First, we contend that for all $l \in [v]$ and $a \in \mathbb{F}_q^{*}$, we have $x_l-a \not\in I$. Assume for contradiction that $x_l-a \in I$ for some $l \in [v]$ and $a \in \mathbb{F}_q^{*}$. Since $I$ is homogeneous per Proposition \ref{characterization of soc-1 ideals}, we have $-a \in I$. This implies $I=\mathcal{R}$, a contradiction. 

  Let $n=k$, $G=\mathcal{R}_{\times}$, $G_i = \mathbb{F}_q^{*}(1+I)$ for all $1 \le i \le k$, and $H$ represent the multiplicative subgroup of $\mathcal{R}_{\times}$ generated by $\{x_l - a \mid a \in \mathbb{F}_q^{*}, l \in T_0\}$. By Lemma \ref{mult chars counting lemma} with these selections, the desired probability is \begin{align*} \frac{1}{(|G|/|\mathbb{F}_q^{*}(1+I)|)^k} \cdot \frac{|H \cap K|}{|H|}\llbracket \chi(H \cap K) = 1 \rrbracket \cdot \prod_{i=1}^k \frac{|G|}{|G_i|} &= \frac{|H \cap K| \cdot \llbracket \chi(H \cap K) = 1 \rrbracket \cdot |\mathbb{F}_q^{*}(1+I)|^k}{|H| \cdot \prod_{i=1}^k |G_i|} \\ &= \frac{|H \cap K| \cdot \llbracket \chi(H \cap K) = 1 \rrbracket}{|H|}, \end{align*} where the total count $|G|/|\mathbb{F}_q^{*}(1+I)|)^k$ comes from the fact that there are $|G|/|\mathbb{F}_q^{*}(1+I)|$ characters $\chi$ with $\chi(1+I) = 1$. Since $x_l - a \not\in I$ for all $a \in \mathbb{F}_q^{*}$ and $l \in T_0$, we have $H \cap K = \{1\}$ and $\llbracket \chi(H \cap K) = 1 \rrbracket = 1$. Therefore, the desired probability becomes $\frac{1}{|H|}$, so it remains to show that $|H| \ge \left(\sum_{k=1}^{\lfloor p/2 \rfloor} \binom{q}{k} \cdot \left(\frac{n}{p}\right)^k\right)^{-|T_0|}$. We require the following technical claims.

  \begin{claim}\label{orders in R}
    In the multiplicative group $\mathcal{R}_{\times}$, we have $\operatorname{ord}(x_l-a) = p^{\lceil \log_p(n) \rceil}$ for all $a \in \mathbb{F}_q^{*}$ and $l \in T_0$.
  \end{claim}
  \begin{proof}[Proof of claim]
    For each $j \ge 1$, we have \[ (x_l-a)^{p^j} = x_l^{p^j} - a^{p^j} \] by the freshman's dream identity. Setting $j=\lceil \log_p(n) \rceil$ returns $(x_l-a)^{p^j} \in \mathbb{F}_q^{*}$ over $\mathcal{R}$, whereas setting $j=\lceil \log_p(n) \rceil-1$ returns $(x_l-a)^{p^j} \not\in \mathbb{F}_q^{*}$ over $\mathcal{R}$. This implies the claim. 
  \end{proof}
  \begin{claim}\label{p-independence of linear factors}
    Let $k \le p-1$ be a positive integer,  $a_1, \dots, a_k$ be distinct elements of $\mathbb{F}_q^{*}$, and $n_1, \dots, n_k$ be positive integers less than $p^{\lceil \log_p(n) \rceil}$. Then \[ \prod_{i=1}^k (1+a_ix_1)^{n_i} \neq 1 \] over $\mathcal{R}$. 
  \end{claim}
  \begin{proof}[Proof of claim]
    Assume for contradiction that there exist such $a_i$ and $n_i$ with $\prod_{i=1}^k (1+a_ix_1)^{n_i} = 1$ over $\mathcal{R}$. Without loss of generality, assume that $p \nmid n_1$, as the terms of the product are symmetric in their index and we may scale down the $n_i$ by powers of $p$ per the freshman's dream identity. By Newton's sums\footnote{Note that under characteristic $p$, Newton's sums hold for all power sums up to the $(p-1)$th power.}, we have \[ \sum_{i=1}^k n_i \cdot a_i^j = 0 \] for all $1 \le j \le k$. As such, we have \[ \begin{bmatrix} 1 & 1 & \dots & 1 \\ a_1 & a_2 & \dots & a_k \\ & & \vdots & \\ a_1^{k-1} & a_2^{k-1} & \dots & a_k^{k-1} \end{bmatrix}\begin{bmatrix} a_1n_1 \\ a_2n_2 \\ \vdots \\ a_kn_k \end{bmatrix} = \mathbf{0} \] over $\mathbb{F}_q$. Since the $a_i$ are distinct, \[ \begin{vmatrix} 1 & 1 & \dots & 1 \\ a_1 & a_2 & \dots & a_k \\ & & \vdots & \\ a_1^{k-1} & a_2^{k-1} & \dots & a_k^{k-1} \end{vmatrix} \neq 0, \] so it follows that $p \mid n_i$ for all $1 \le i \le k$. This is a contradiction to the assumption that $p \nmid n_1$. 
  \end{proof}
  \noindent By Claim \ref{orders in R}, we have \begin{align*} |H| \ge \sum_{k=1}^{\lfloor p/2 \rfloor} \binom{q}{k} \cdot \left(p^{\lfloor \log_p(n) \rfloor} -1\right)^k \ge \sum_{k=1}^{\lfloor p/2 \rfloor} \binom{q}{k} \cdot \left(\frac{n}{p}\right)^k \end{align*} as no two polynomials generated by at most $\lfloor p/2 \rfloor$ distinct linear factors each are equal. This completes the proof.
\end{proof}
We require one final technical fact, which regards the distribution of a multiplicative character sum. This is a straightforward application of Claim \ref{p-independence of linear factors}. 
\begin{lemma}\label{expected value of linear factors}
    We have \[ \mathbb{E}_{\chi \in \widehat{\mathcal{R}_{\times}}}\left[\left|\sum_{a \in \mathbb{F}_q^{*}} \chi(x_l-a)\right|\right] \le \frac{q-1}{\sqrt{p-1}}. \]
\end{lemma}
\begin{proof}
    Let $a_1, \dots, a_{p-1}$ be any $p-1$ distinct elements of $\mathbb{F}_q^{*}$. Then we have \[ \mathbb{E}_{\chi \in \widehat{\mathcal{R}_{\times}}}\left[\left|\sum_{k=1}^{p-1} \chi(x_l-a_k)\right|^2\right] = p-1, \] as all $x_l-a_k$ terms are multiplicatively independent in $\mathcal{R}_{\times}$ per Claim \ref{p-independence of linear factors}. By the Cauchy-Schwarz inequality, \[ \mathbb{E}_{\chi \in \widehat{\mathcal{R}_{\times}}}\left[\left|\sum_{k=1}^{p-1} \chi(x_l-a_k)\right|\right]^2 \le \mathbb{E}_{\chi \in \widehat{\mathcal{R}_{\times}}}\left[\left|\sum_{k=1}^{p-1} \chi(x_l-a_k)\right|^2\right] = p-1, \] so by the triangle inequality we have \[ \mathbb{E}_{\chi \in \widehat{\mathcal{R}_{\times}}}\left[\left|\sum_{a \in \mathbb{F}_q^{*}} \chi(x_l-a)\right|\right] \le \frac{q-1}{p-1} \cdot \mathbb{E}_{\chi \in \widehat{\mathcal{R}_{\times}}}\left[\left|\sum_{k=1}^{p-1} \chi(x_l-a_k)\right|\right] \le \frac{q-1}{p-1} \cdot \sqrt{p-1} = \frac{q-1}{\sqrt{p-1}}, \] as desired. 
\end{proof}

\begin{rem*}
    We remark that when $q$ is prime, Lemma \ref{expected value of linear factors} gives a stronger result than that of the Weil bounds. 
\end{rem*}

\section{Li-Wan Sieve Application and Proof of Main Result}\label{section: li-wan sieve}
Now, we introduce the Li-Wan sieve formula. This sieve -- which is primarily used for solving distinct coordinate counting problems -- is an improvement of the inclusion-exclusion principle, in that it uses a significantly lower number of terms. The Li-Wan sieve formula was previously used in determining the distance distribution of Reed-Solomon codes \cite{LiWan2020}; for other applications, we refer the reader to \cite{LiWan2010}, \cite{LI2012170}, and \cite{Zhu2010}. 

\begin{lemma}[\cite{LiWan2010}, Theorem 1.1]
  Let $D$ be a finite set, $n$ be a positive integer, $X$ be a subset of $D^n$, and $f$ be a complex-valued function defined on $X$. Denote $\overline{X}$ by the set of elements of $X$ with all coordinates pairwise distinct. Given a permutation $\tau \in S_n$, let $X_{\tau}$ denote the set of elements of $X$ which are fixed by permuting the coordinates in accordance to $\tau$. Let \[ F_{\tau} := \sum_{(x_1, \dots, x_n) \in X_{\tau}} f(x_1, \dots, x_n). \] Then \[ \sum_{(x_1, \dots, x_n) \in \overline{X}} f(x_1, \dots, x_n) = \sum_{\tau \in S_n} \operatorname{sgn}(\tau)F_{\tau}, \] where $\operatorname{sgn}(\tau)$ is the sign of the permutation $\tau$. In particular, \[ \left|\sum_{(x_1, \dots, x_n) \in \overline{X}} f(x_1, \dots, x_n)\right| \le \sum_{\tau \in S_n} |F_{\tau}|. \]
\end{lemma}
We are ready to prove our main result. 
\begin{proof}[Proof of Theorem \ref{main result}]
In this proof, we use $\alpha$ to denote the tuple $(a_1, \dots, a_v) \in \mathbb{F}_q^v \setminus \{\mathbf{0}\}$. For each permutation $\tau$ of $[k]$, consider the sum \[ G_{\tau} := \sum_{(\chi_i) \in (\widehat{\mathcal{E}_{\times}})^k }\left|\sum_{(\alpha_1, \dots, \alpha_k) \in X_{\tau}} \prod_{i=1}^k [\chi_i(\Red(P_{\alpha_i})) \cdot G(\chi_i, \psi_{S(P_{\alpha})})]\right|. \] By the Li-Wan sieve formula, we are required to bound $\frac{q^k}{|\mathcal{E}|^k} \cdot \sum_{\tau \in S_k} |G_{\tau}|$.

Let $\tau$ be a fixed permutation in $S_k$ which is composed of cycles $C_1, \dots, C_r$. Then, \begin{align*} |G_{\tau}| &= \prod_{j=1}^r \left|\sum_{\alpha \in \mathbb{F}_q^v \setminus \{\mathbf{0}\}} \left(\prod_{i \in C_j} \prod_{a_l \neq 0} \chi_{i}(1-x_la_l)\right)\left(\prod_{i \in C_j} G(\chi_i, \psi_{S(P_{\alpha})})\right)\right| \\ &\le \prod_{j=1}^r \sum_{\varnothing \subset T \subseteq [v]} \left|\sum_{\substack{\alpha \in \mathbb{F}_q^v \\ a_l \neq 0 \iff l \in T}} \left(\prod_{i \in C_j}\prod_{l \in T} \chi_i(1-x_la_l)\right)\right|\left|\prod_{i \in C_j} G(\chi_i, \psi_{x_T})\right|. \end{align*} Denote $q_1 := \min(\ell-1, (q-1)/\sqrt{q(p-1)})\sqrt{q}$. Observe that \begin{align*} \mathbb{E}_{\chi \in \widehat{\mathcal{E}_{\times}}}\left|\sum_{a \in \mathbb{F}_q^{*}} \chi(1-ax_l)\right| \le q_1 \end{align*} by the Weil bound \cite[, Theorem 2.1]{Wan1998} and Lemma \ref{expected value of linear factors}. For all $l \in [v]$, sequences $(\chi_i)_{i=1}^k$ of characters defined on $\mathbb{F}_q[x_l]$ that are periodic modulo $x_l^{\ell}$, and nonempty subsets $C$ of $[k]$, set \[ m(l, C) := 1 - \left\llbracket\prod_{i \in C} \chi_i(x_l-a) = 1 \quad \forall a \in \mathbb{F}_q[x_l]\right\rrbracket. \] Then we have \begin{align*} \left|\sum_{\substack{\alpha \in \mathbb{F}_q^v \\ a_l \neq 0 \iff l \in T}} \left(\prod_{i \in C_j}\prod_{l \in T} \chi_i(1-x_la_l)\right)\right| &\le \prod_{l \in T} [q_1^{m(l, C_j)}(\tfrac{q-1}{\sqrt{p-1}})^{1-m(l, C_j)}] \\ &= q_1^{\sum_{l \in T} m(l, C_j)}(\tfrac{q-1}{\sqrt{p-1}})^{|T| - \sum_{l \in T} m(l, C_j)}. \end{align*} Therefore, by the Li-Wan sieve formula, the desired sum is at most the sum of \[ \left(\frac{q}{|\mathcal{E}|}\right)^k\sum_{\tau \in S_k} \prod_{j=1}^r \sum_{T \subseteq [v]} q_1^{\sum_{l \in T} m(l, C_j)}(\tfrac{q-1}{\sqrt{p-1}})^{|T| - \sum_{l \in T} m(l, C_j)}\left|\prod_{i \in C_j} G(\chi_i, \psi_{x_T})\right|. \] over all characters $(\chi_i)$ and $\psi \neq \mathbf{1}$. By Proposition \ref{gauss sum size formula}, the above sum simplifies to the sum of \begin{equation}\label{sieve application sum} q^k \sum_{\tau \in S_k} \prod_{j=1}^r \sum_{T \subseteq [v]} q_1^{\sum_{l \in T} m(l, C_j)}(\tfrac{q-1}{\sqrt{p-1}})^{|T| - \sum_{l \in T} m(l, C_j)} \cdot (q^{-|C_j|/2})^{\operatorname{codim}(K_{\psi_{x_T}})} \end{equation} over all characters $(\chi_i)$ and $\psi \neq \mathbf{1}$ that satisfy $\chi_i(1+K_{\psi_{x_T}})=1$ and $\chi_i(\mathbb{F}_q^{*}) = 1$. We will instead evaluate the sum at individual cyclic permutations $\tau$ that act non-trivially on a set $C$, and subsequently compute the sum over all permutations of $[k]$ using exponential generating functions. 
\begin{claim}\label{sum over multiplicative chars}
  Let $C$ be a nonempty subset of $[k]$, and fix a non-trivial additive character $\psi$. Then the sum $\Sigma(C; \psi)$ of \[ q_1^{\sum_{l \in T} m(l, C)}(\tfrac{q-1}{\sqrt{p-1}})^{|T| - \sum_{l \in T} m(l, C)}\] over all nonempty subsets $T$ of $[v]$ and characters $(\chi_i)$ that satisfy $\chi_i(1+K_{\psi_{x_T}})=1$ and $\chi_i(\mathbb{F}_q^{*}) = 1$ is given by \[ \sum_{\varnothing \subset T \subseteq [v]} q_2^{|T|}\left(\frac{|\mathcal{E}|}{|K_{\psi_{x_T}}|}\right)^{|C|}. \]
\end{claim}
\begin{proof}[Proof of claim]
  Let us begin by fixing $\varnothing \subset T \subseteq [v]$ and computing the expected value $\mathbb{E}_T(C; \psi)$ of \[ q_1^{\sum_{l \in T} m(l, C)}(\tfrac{q-1}{\sqrt{p-1}})^{|T| - \sum_{l \in T} m(l, C)} = \prod_{l \in T} q_1^{m(l, C)}(\tfrac{q-1}{\sqrt{p-1}})^{1-m(l, C)}. \] over characters $(\chi_i)$ satisfying the given conditions and chosen uniformly at random. By Proposition \ref{mult chars probability}, each of the variables $m(l, C)$ for $l \in [v]$ is $0$ with probability at most $\left(\sum_{k=1}^{\lceil p/2 \rceil} \binom{q}{k} \cdot (n/p)^k\right)^{-1}$. Hence, it follows that \begin{align*} \mathbb{E}[q_1^{m(l, C)}(\tfrac{q-1}{\sqrt{p-1}})^{1-m(l, C)}] &\le q_2 \end{align*} for each $l \in [v]$. Proposition \ref{mult chars probability} also implies that the variables $m(l, C)$ are independent across $l \in [v]$, so we conclude that \[ \mathbb{E}_T(C; \psi) = \prod_{l \in T} \mathbb{E}[q_1^{m(l, C)}(\tfrac{q-1}{\sqrt{p-1}})^{1-m(l, C)}] = q_2^{|T|}. \] Multiplying by the number of characters satisfying the given and summing over $T$, we have \[ \Sigma(C; \psi) = \sum_{\varnothing \subset T \subseteq [v]} q_2^{|T|}\left(\frac{|\mathcal{E}|}{|K_{\psi_{x_T}}|}\right)^{|C|}, \] as desired.
\end{proof}
By Claim \ref{sum over multiplicative chars}, the desired sum is at most \[ \sum_{\substack{\psi \neq \mathbf{1}}}\sum_{\varnothing \subset T \subseteq [v]} q_2^{|T|}\left(\frac{|\mathcal{E}|}{|K_{\psi_{x_T}}|}\right)^{|C|} \cdot \left[q^k (q^{-|C|/2})^{\operatorname{codim}(K_{\psi_{x_T}})}\right]. \] By Lemma \ref{counting additive chars}, the above sum is at most \begin{align*}
    & q^k \left(1-\frac{1}{q}\right)^2 \cdot \sum_{\varnothing \subset T \subseteq [v]} q_2^{|T|}\left(F_{\ell-|T|}(q^{|C|/2}) + \frac{q}{(q-1)^2} \cdot q^{\binom{\ell+v-3}{v-1}-\binom{\ell-|T|+v-2}{v-1}}\right). 
\end{align*} 
Using Proposition \ref{ideals and LC bijection} in conjunction with geometric series type bounding where the first term corresponds to the sum over subsets $T$ of size $1$, the above is at most \[ \mathcal{P}(|C|) := \frac{vq_2q^{\binom{\ell+v-3}{v-1}}}{(q-1)^2} \cdot q^{\binom{\lceil \ell /2 \rceil+v}{v} \cdot |C|}. \] We will now derive an upper bound on the desired sum over all permutations $\tau$. This is at most \[ q^k\sum_{\substack{\tau \in S_k \\ \tau = C_1 \circ \dots \circ C_r}}  \frac{1}{r!}\prod_{i=1}^r \mathcal{P}(|C_i|). \] Using exponential generating functions, we will find a closed form for the above. Observe that \begin{align*} q^k\sum_{\substack{\tau \in S_k \\ \tau = C_1 \circ \dots \circ C_r}}  \frac{1}{r!}\prod_{i=1}^r \mathcal{P}(|C_i|) &= q^k[t^k] \exp\left(\sum_{n \ge 1} \frac{\mathcal{P}(n)}{n}t^n\right) \\ &= q^k[t^k] \exp\left(-\frac{vq_2q^{\binom{\ell+v-3}{v-1}}}{(q-1)^2}\log(1-q^{\binom{\lceil \ell /2 \rceil+v}{v}}t)\right) \\ &= q^k[t^k] \left(1-q^{\binom{\lceil \ell /2 \rceil+v}{v}}t\right)^{-\frac{vq_2q^{\binom{\ell+v-3}{v-1}}}{(q-1)^2}} \\ &= \binom{\frac{vq_2q^{\binom{\ell+v-3}{v-1}}}{(q-1)^2}+k-1}{k-1} \cdot q^{k \left(\binom{\lceil \ell /2 \rceil+v}{v}+1\right)}. \end{align*} This completes the proof. 
\end{proof}

\section{Future Directions}\label{section: future directions}
In this section, we provide several research questions that allow for potential extensions of this work. 

The simplest type of variant of the main problem we consider are those involving prescribed evaluation vectors. This variant uses nearly the same character sum formula as for prescribed zeroes. 
\begin{qn}
    Use our approach to count polynomials with prescribed evaluation vectors, coefficients, and zeroes (e.g.~polynomials whose evaluations are all squares in $\mathbb{F}_q$).
\end{qn}
A second possible generalization of this work is to consider a different quotient ring. While this generalization does not directly apply to coding theory, it is another type of polynomial enumeration problem to which a similar approach can apply. We note that our results on ideal enumeration in truncated polynomial quotient rings are very specific to those quotient rings, which makes this generalization worthwhile. 
\begin{qn}
Generalize the main result to quotient rings of the form $\mathbb{F}_q[x_1, \dots, x_v]/I_0$, where $I_0$ is an ideal of $\mathbb{F}_q[x_1, \dots, x_v]$. 
\end{qn}
Another limitation of this work (which is also present in \cite{Gao2022, LiWan2020}) is that we require $D$ to be the same (or nearly the same) as $\mathbb{F}_q^v$, or else we cannot factor some of our character sums into character sums on linear polynomials. This prompts us to ask the following. 
\begin{qn}
    What types of approaches can be implemented when $D$ deviates significantly from $\mathbb{F}_q^v$? 
\end{qn}
Finally, we ask if the basis of Lagrange-interpolating polynomials we consider can be made more efficient.
\begin{qn}
    We characterized polynomials with prescribed vanishing set using Lagrange interpolation. Can a stronger characterization (perhaps borrowing ideas from algebraic geometry) be used to strengthen our error bounds?
\end{qn}

\section*{Acknowledgements}
I would like to thank Dr.\ Simon Rubinstein-Salzedo for several valuable discussions during the creation of this paper. I would also like to thank Dr.\ Harold Polo for discussions and research directions that inspired this paper.

\printbibliography

\end{document}